\documentclass[11pt, oneside]{amsart}   	
\allowdisplaybreaks
\usepackage{geometry}                		
\usepackage{graphicx}				
\usepackage{amssymb}

\newtheorem{theorem}{\bf Theorem}[section]
\newtheorem{proposition}[theorem]{\bf Proposition}
\newtheorem{lemma}[theorem]{\bf Lemma}
\newtheorem{corollary}[theorem]{\bf Corollary}

\newtheorem{example}[theorem]{\bf Example}

\newtheorem{definition}[theorem]{\bf Definition}
\newtheorem{remark}[theorem]{\bf Remark}

\newtheorem{problem}[theorem]{\bf Problem}

\newenvironment{proofof}[1]{\noindent{\it Proof of
#1.}}{\hfill$\square$\\\mbox{}}

\title{Symmetric polynomials over finite fields}
\author[M\'aty\'as Domokos]
{M\'aty\'as Domokos}
\address{Alfr\'ed R\'enyi Institute of Mathematics,
Re\'altanoda utca 13-15, 1053 Budapest, Hungary,
ORCID iD: https://orcid.org/0000-0002-0189-8831}
\email{domokos.matyas@renyi.hu}
\author[Botond Mikl\'osi]{Botond Mikl\'osi} 
\address{E\"otv\"os Lor\'and University, 
P\'azm\'any P\'eter s\'et\'any 1/C, 1117 Budapest, Hungary} 
\email{miklosiboti@gmail.com}

\begin{document}
\thanks{Partially supported by the Hungarian National Research, Development and Innovation Office,  NKFIH K 138828,  K 132002.}
\subjclass[2010]{Primary 13A50; Secondary 12E20}
\keywords{separating sets, elementary symmetric polynomials, finite fields}

\begin{abstract} 
It is shown that two vectors with coordinates in the finite $q$-element field of characteristic $p$ belong to the same orbit under the natural action of the symmetric group if each of the elementary symmetric polynomials of degree $p^k,2p^k,\dots,(q-1)p^k$, $k=0,1,2,\dots$ has the same value on them. This separating set of polynomial invariants for the natural permutation representation  of the symmetric group is not far from being minimal when 
$q=p$ and the dimension is large compared to $p$. A relatively small separating set of multisymmetric polynomials over the field of $q$ elements is derived. 
\end{abstract} 

\maketitle

\section{Introduction}\label{sec:intro} 

Throughout this paper $F$ stands for an arbitrary field, $q$ stands for a power of a prime $p$, and $\mathbb{F}_q$ stands for the field of $q$ elements. 
The symmetric group $S_n$ acts on the vector space $F^n$ by permuting coordinates: for $\pi\in S_n$ and $v=(v_1,\dots,v_n)\in F^n$ we have 
$\pi\cdot v=(v_{\pi^{-1}(1)},\dots,v_{\pi^{-1}(n)})$. Denote by $x_1,\dots,x_n$ the basis of the dual space of $F^n$ dual to the standard basis in $F^n$. We have an induced action of $S_n$ via $F$-algebra automorphisms on the polynomial algebra $F[x_1,\dots,x_n]$. 
In particular, $\pi\cdot x_i=x_{\pi(i)}$ for $\pi\in S_n$ and $i\in\{1,\dots,n\}$. 
The algebra of $S_n$-invariant polynomials is 
$F[x_1,\dots,x_n]^{S_n}=\{f\in F[x_1,\dots,x_n]\mid \forall \pi\in S_n:\ \pi\cdot f=f\}$. 
A subset $T$ of $F[x_1,\dots,x_n]^{S_n}$ is \emph{separating} if for any $v,w\in F^n$ with  different $S_n$-orbit there exists an element $f\in T$ such that $f(v)\neq f(w)$. 
This is a special case of the notion of \emph{separating set of polynomial invariants} of  
(not necessarily finite) groups; for the general notion and  basic facts about it we refer to \cite[Section 2.4]{derksen-kemper}. By a \emph{minimal separating set} we shall mean a  separating set none of whose proper subsets are separating (i.e. a separating set minimal with respect to inclusion). It is well known that every separating set has a finite separating subset (by a straightforward modification of the proof of \cite[Theorem 2.4.8]{derksen-kemper}), therefore any separating set contains a minimal separating subset, and a minimal separating set is necessarily finite. On the other hand, different minimal separating sets 
may have different cardinality, so a minimal separating set does not necessarily have minimal possible cardinality.

It is well known that the algebra $F[x_1,\dots,x_n]^{S_n}$ is minimally generated by the \emph{elementary symmetric polynomials} 
$s^{(n)}_k=\sum_{1\le i_1<\cdots<i_k\le n}x_{i_1}\cdots x_{i_k}$. 
Moreover, when $F$ is algebraically closed (or $F=\mathbb{R})$, then $\{s^{(n)}_k\mid k=1,\dots,n\}$ is a separating set in  $F[x_1,\dots,x_n]^{S_n}$ having the least possible number of elements (in particular, it is a minimal separating set). 
However, when $F$ is finite, in general the above separating set is not even minimal with respect to inclusion.  For a real number $v$ we shall write $\lfloor v\rfloor$ for the largest integer not strictly bigger than $v$, and write $\lceil v\rceil$ for the smallest integer which is not  strictly smaller than $v$. 
Kemper, Lopatin and Reimers \cite[Lemma 4.3]{KLR} proved that 
$\{s^{(n)}_{2^k}\mid k=0,1,\dots,\lfloor \log_2 n\rfloor\}$ 
is a minimal separating set in $\mathbb{F}_2[x_1,\dots,x_n]^{S_n}$ (and this separating set has the least possible number of elements). 
Our aim in this note is to extend the above result from the $2$-element field $\mathbb{F}_2$ to any finite field $\mathbb{F}_q$. 
For a positive integer $n$ set  
\[[n]_q:=\{jp^k \mid j\in\{1,\dots,q-1\}, \quad k\in \mathbb{Z}_{\ge 0}, \quad jp^k\le n\}.\] 
We shall prove the following: 

\begin{theorem}\label{thm:main} 
The elementary symmetric polynomials $s^{(n)}_{m}$ 
with $m\in [n]_q$ 
form a separating subset in $\mathbb{F}_q[x_1,\dots,x_n]^{S_n}$.  
\end{theorem} 

While the study of separating sets of polynomial invariants of groups became rather popular in the past two decades, 
as far as we know, the recent paper \cite{KLR} is the first studying separating sets of polynomial invariants over finite fields; see also that paper for motivation (for example, for the connection to the graph isomorphism problem). 
On the other hand, 
Theorem~\ref{thm:main} has an equivalent reformulation not referring to the notion of separating sets of polynomial invariants, but as a statement about univariate polynomials over finite fields as follows: 

\begin{theorem}\label{thm:main-reformulation} 
Let $f,g\in \mathbb{F}_q[x]$ be monic polynomials of degree $n$, such that both $f$ and $g$ split as a product of root factors over $\mathbb{F}_q$. 
Assume that for all $j\in [n]_q$, the degree $n-j$ coeficient of $f$ coincides with the degree $n-j$ coefficient of $g$. Then we have $f=g$.  
\end{theorem} 

We also investigate wether the separating set given in Theorem~\ref{thm:main} is minimal. It turns out that it is minimal (with respect to inclusion)  
for $q=3,4,5$ with arbitrary $n$ and for $q=7$ with ``most" $n$ (see Corollary~\ref{cor:p=3,5,7}). 
However, computer calculations show that it is not always minimal (see the case $q=7$, $n=5$ in Corollary~\ref{cor:p=3,5,7},  or the results in Section~\ref{subsec:p=11} for $p=11$). On the other hand, we point out in Proposition~\ref{prop:d(n)} that when $q=p$ and $n$ is large compared to $p$, then in a certain sense, the separating set given in Theorem~\ref{thm:main} is not far from being minimal. 

In Section~\ref{sec:multisymmetric} we turn to the study of multisymmetric polynomials over 
the field $\mathbb{F}_q$. Separating sets of multisymmetric polynomials are studied in 
\cite{reimers}, \cite{LR}, and a minimal separating set of multisymmetric polynomials over $\mathbb{F}_2$ is given in \cite[Theorem 4.8]{KLR}. 
Here we shall exploit Theorem~\ref{thm:main} to obtain a relatively small separating set of multisymmetric polynomials over $\mathbb{F}_q$ for an arbitrary prime power $q$ in Theorem~\ref{thm:multisymmetric}. 

\section{Preliminaries on polynomials} 

\begin{lemma}\label{lemma:distinct roots} 
Let $F$ be an arbitrary field, $f=\sum_{i=0}^dc_ix^i\in F[x]$ a polynomial whose formal derivative $f'$ is not zero (i.e. $c_i\neq 0$ for some $i$ not divisible by 
the characteristic of $F$). Assume that $c_0\neq 0$ and $c_1=\dots=c_m=0$. 
Then $f$ has at least $m+1$ distinct roots in the algebraic closure of $F$. 
\end{lemma} 

\begin{proof}
The formal derivative of $f$ is 
\[f'=(m+1)c_{m+1}x^m+\sum_{j=m+2}^djc_jx^{j-1}=x^mh\] 
for some non-zero polynomial $h\in F[x]$. 
We have 
\begin{equation}\label{eq:deg(h)}\deg(h)=\deg(f')-m\le \deg(f)-1-m.
\end{equation}  
Recall that the number of distinct roots of $f$ in the algebraic closure of $F$ is greater than or equal to the difference of the degree of $f$ and the degree of the greatest common divisor 
$\mathrm{gcd}_{F[x]}(f,f')$
of $f$ and $f'$. As $x$ does not divide $f$ 
(recall that $c_0\neq 0$), we have 
\[\mathrm{gcd}_{F[x]}(f,f')=\mathrm{gcd}_{F[x]}(f,h).\]
Consequently, 
\[\deg(\mathrm{gcd}_{F[x]}(f,f'))\le \deg(h) \text{ and } \deg(f)-\deg(\mathrm{gcd}_{F[x]}(f,f'))
\ge \deg(f)-\deg(h).\]
It follows by \eqref{eq:deg(h)} that the number of distinct roots of $f$ in the algebraic closure of $f$ is at least 
\[\deg(f)-\deg(h)\ge m+1.\] 
\end{proof} 

We shall denote by $\mathbb{F}_q^{\times}$ the set of non-zero elements in $\mathbb{F}_q$

\begin{corollary} \label{cor:i_1=0}
Given a map $\mathcal{O}:\mathbb{F}_q^{\times} \to \mathbb{Z}_{\ge 0}$ 
consider the polynomial 
\[G_{\mathcal{O}}(x):=\prod_{a\in \mathbb{F}_q^{\times}}(1+ax)^{\mathcal{O}(a)}\in \mathbb{F}_q[x].\] 
Assume that all terms of $G_{\mathcal{O}}(x)$ of degree $1,2,\dots,q-1$ have coefficient zero. Then $p$ divides $\mathcal{O}(a)$ for all $a\in \mathbb{F}_q^{\times}$. 
\end{corollary} 

\begin{proof} 
Suppose for contradiction that $p$ does not divide $\mathcal{O}(a)$ for some $a\in \mathbb{F}_q^{\times}$. 
Then the formal derivative $G_{\mathcal{O}}'$ of $G_{\mathcal{O}}$ is not the zero polynomial in $\mathbb{F}_q[x]$. 
Thus Lemma~\ref{lemma:distinct roots} applies for $G_{\mathcal{O}}$, and we conclude that $G_{\mathcal{O}}$ has at least $(q-1)+1=q$ distinct roots in the algebraic closure of $\mathbb{F}_q$. However, 
$G_{\mathcal{O}}$ splits as a product of root factors already over $\mathbb{F}_q$, 
and all its roots are non-zero, so 
$G_{\mathcal{O}}$ has at most $q-1$ distinct roots in $\mathbb{F}_q$ (and hence in its algebraic closure), a contradiction. 
\end{proof}

\begin{lemma} \label{lemma:i_1=j_1} 
Take two maps $\mathcal{O},\mathcal{P}:\mathbb{F}_q^{\times}\to\{0,1,\dots,q-1\}$ and consider the polynomials 
\[G_{\mathcal{O}}(x):=\prod_{a\in \mathbb{F}_q^{\times}}(1+ax)^{\mathcal{O}(a)}=\sum_jb_jx^j 
\quad \text{ and } \quad 
G_{\mathcal{P}}(x):=\prod_{a\in \mathbb{F}_q^{\times}}(1+ax)^{\mathcal{P}(a)}
=\sum_jc_jx^j.\]  
Suppose that $b_1=c_1,b_2=c_2,\dots,b_{q-1}=c_{q-1}$. 
Then we have $\mathcal{O}(a)\equiv \mathcal{P}(a)$ modulo $p$ for all $a\in \mathbb{F}_q^{\times}$. 
\end{lemma}

\begin{proof} 
By assumption the polynomial $G_{\mathcal{O}}-G_{\mathcal{P}}$ is divisible in $\mathbb{F}_q[x]$ by $x^q$. Denote by $D$ the greatest common divisor in $\mathbb{F}_q[x]$ of  
$G_{\mathcal{O}}$ and $G_{\mathcal{P}}$. 
Set $g:=G_{\mathcal{O}}/D$ and $h:=G_{\mathcal{P}}/D$, so $x^q$ divides $(g-h)D$. 
As $x$ does not divide $G_{\mathcal{O}}$, it does not divide $D$, and therefore  $x^q$ divides $g-h$ in $\mathbb{F}_q[x]$. There exist some disjoint subsets $A,B$ of $\mathbb{F}_q^{\times}$ such that 
\[g=\prod_{a\in A}(1+ax)^{k_a}\ \text{ and }\ 
h=\prod_{b\in B}(1+bx)^{m_b},\] 
where $k_a$, $m_b$ are positive integers less than or equal to $q-1$. 
Set 
\[f:=g\prod_{b\in B}(1+bx)^{q-m_b}.\] 
Then 
\[f=(g-h)\prod_{b\in B}(1+bx)^{q-m_b}+\prod_{b\in B}(1+bx)^q.\] 
The first summand on the right hand side above is divisible by $x^q$, whereas the second summand minus $1$ is also divisible by $x^q$. This implies that $f-1$ is divisible by $x^q$. 
On the other hand, we have 
\[f=\prod_{a\in A}(1+ax)^{k_a}\prod_{b\in B}(1+bx)^{q-m_b}.\] 
By Corollary~\ref{cor:i_1=0} we conclude that $p$ divides $k_a$ for all  $a\in A$ 
and $p$ divides $q-m_b$  (and hence $p$ divides $m_b$) for all $b\in B$. 
Define $\mathcal{R}:\mathbb{F}_q^{\times}\to \{0,1,\dots,q-1\}$ by 
\[D=\prod_{c\in \mathbb{F}_q^{\times}}(1+cx)^{\mathcal{R}(c)}.\] 
Then 
\[\mathcal{O}(c)=\begin{cases} 
\mathcal{R}(c)+k_c &\text{ for }c\in A \\
\mathcal{R}(c) &\text{ for }c\notin A\end{cases}
\qquad\text{ and  }\qquad\mathcal{P}(c)=
\begin{cases} 
\mathcal{R}(c)+m_c &\text{ for }c\in B \\
\mathcal{R}(c) &\text{ for }c\notin B.\end{cases}\]
 As $p$ divides $k_a$ and $m_b$ for all $a\in A$ and $b\in B$, this clearly implies that 
 both $\mathcal{O}(c)$ and $\mathcal{P}(c)$ are congruent to $\mathcal{R}(c)$ modulo $p$, 
hence  $p$ divides $\mathcal{O}(c)-\mathcal{P}(c)$ for all $c\in \mathbb{F}_q^{\times}$. 
\end{proof} 


\section{A separating set of elementary symmetric polynomials} \label{sec:a separating set} 

Denote by $\Pi_{q,n}$ the set of functions $\mathcal{O}:\mathbb{F}_q^{\times}\to \mathbb{Z}_{\ge 0}$ satisfying $\sum_{a\in\mathbb{F}_q^{\times}}\mathcal{O}(a)\le n$. 
There is a natural bijection between $\Pi_{q,n}$ and the set of $S_n$-orbits in $\mathbb{F}_q^n$; namely, associate with $\mathcal{O}\in \Pi_{q,n}$ the set of vectors in $\mathbb{F}_q^n$ having $\mathcal{O}(a)$ coordinates equal to $a$ for each $a\in \mathbb{F}_q^{\times}$, and having $n-\sum_{a\in\mathbb{F}_q^{\times}}\mathcal{O}(a)$ zero coordinates. 
We shall write $s_k(\mathcal{O})$ for the value of the elementary symmetric polynomial 
$s^{(n)}_k\in \mathbb{F}_q[x_1,\dots,x_n]$ on the vectors in $\mathbb{F}_q^n$ that belong to the orbit labelled by $\mathcal{O}$. For $k>n$ we set $s_k(\mathcal{O})=0$, and set $s_0(\mathcal{O})=1$. 

For $\mathcal{O}\in \Pi_{q,n}$ and $a\in \mathbb{F}_q^{\times}$ we have 
\[\mathcal{O}(a)=\sum_{j=0}^{\lfloor \log_p n\rfloor}\mathcal{O}(a)_jp^j\] 
for some uniquely determined integers $\mathcal{O}(a)_j\in\{0,1,\dots,p-1\}$ 
(i.e. the numbers $\mathcal{O}(a)_j$ are the digits of the non-negative integer $\mathcal{O}(a)$ in the number system with base $p$). 
For $k=0,1,\dots,\lfloor \log_p n\rfloor$ denote by 
$\mathcal{O}_{\{k\}}\in \Pi_{q,n}$ the function given by 
\[\mathcal{O}_{\{k\}}(a)=\sum_{j=0}^k\mathcal{O}(a)_jp^j, \quad a\in \mathbb{F}_q^{\times}\] 
and let $\mathcal{O}^{[k]}\in \Pi_{q,n}$ be the function given by 
\[\mathcal{O}^{[k]}(a)=\mathcal{O}(a)-\mathcal{O}_{\{k\}}(a), \quad a\in \mathbb{F}_q^{\times},\]
whereas 
$\mathcal{O}^{[k]/p^{k+1}}\in \Pi_{q,n}$ is the function given by  
\[\mathcal{O}^{[k]/p^{k+1}}(a)=\frac{\mathcal{O}^{[k]}(a)}{p^{k+1}}
, \quad a\in \mathbb{F}_q^{\times}.\]

For $\mathcal{O}\in \Pi_{q,n}$ set 
\begin{equation}\label{eq:G_O}
G_{\mathcal{O}}(x):=\sum_{j=0}^ns_j(\mathcal{O})x^j=
\prod_{a\in \mathbb{F}_q^{\times}}(1+ax)^{\mathcal{O}(a)}\in \mathbb{F}_q[x].
\end{equation}  
With this notation we have the obvious equalities 
\begin{equation}\label{eq:G-factorization} 
G_{\mathcal{O}}(x)=G_{\mathcal{O}_{\{k\}}}(x)\cdot G_{\mathcal{O}^{[k]}}(x)\in \mathbb{F}_q[x]
\end{equation}
and 
\begin{equation} \label{eq:G-p-power}
G_{\mathcal{O}^{[k]}}(x)=G_{\mathcal{O}^{[k]/p^{k+1}}}(x)^{p^{k+1}}. 
\end{equation} 

\begin{lemma}\label{lemma:0} 
Suppose that $\mathcal{O},\mathcal{P}\in \Pi_{q,n}$ satisfy 
\[s_j(\mathcal{O})=s_j(\mathcal{P}) \text{ for }j=1,2,\dots,q-1.\] 
Then $\mathcal{O}_{\{0\}}=\mathcal{P}_{\{0\}}$ (i.e. $\mathcal{O}(a)$ is congruent to 
$\mathcal{P}(a)$ modulo $p$ for all $a\in \mathbb{F}_q^{\times}$). 
\end{lemma} 

\begin{proof} 
We have $q=p^{e+1}$ for some nonnegative integer $e$. 
By \eqref{eq:G-factorization} and \eqref{eq:G-p-power} we have 
\[\mathcal{G}_{\mathcal{O}}(x)=G_{\mathcal{O}_{\{e\}}}(x)\cdot G_{\mathcal{O}^{[e]/q}}(x)^q.\] 
The coefficient of $x^j$ in $G_{\mathcal{O}^{[e]/q}}(x)^q$ is non-zero only if $q$ divides $j$, and the constant term of $G_{\mathcal{O}^{[e]/q}}(x)^q$ is $1$. 
Comparing the degree $j$ coefficients of the two sides of the above equality for $j=1,\dots,q-1$ we get 
\begin{equation}\label{eq:s(O_{k})}s_j(\mathcal{O})=s_j(\mathcal{O}_{\{e\}}) \text{ for }j=1,\dots,q-1.\end{equation} 
Similarly we have 
\begin{equation} \label{eq:s(P_{k})} 
s_j(\mathcal{P})=s_j(\mathcal{P}_{\{e\}}) \text{ for }j=1,\dots,q-1.\end{equation} 
Note that $\mathcal{O}_{\{e\}}$ and $\mathcal{P}_{\{e\}}$ are maps from $\mathbb{F}_q^{\times}$ to $\{0,1,\dots,q-1\}$.  
Moreover, for $j=1,\dots,q-1$ the equality $s_j(\mathcal{O})=s_j(\mathcal{P})$ 
implies by \eqref{eq:s(O_{k})} and \eqref{eq:s(P_{k})} that the degree $j$ coefficient 
of $G_{\mathcal{O}_{\{e\}}}(x)$ coincides with the degree $j$ coefficient of 
$G_{\mathcal{P}_{\{e\}}}(x)$. 
Consequently, Lemma~\ref{lemma:i_1=j_1} applies for $\mathcal{O}_{\{e\}}$ and 
$\mathcal{P}_{\{e\}}$, and yields the desired equality 
$\mathcal{O}_{\{0\}}=\mathcal{P}_{\{0\}}$.  
\end{proof} 

\begin{lemma}\label{lemma:inductive step} 
Suppose that $\mathcal{O},\mathcal{P}\in \Pi_{q,n}$ and for some 
$k\in \{0,1,\dots,\lfloor \log_p n\rfloor-1\}$ 
we have $\mathcal{O}_{\{k\}}=\mathcal{P}_{\{k\}}$ and 
$s_{jp^{k+1}}(\mathcal{O})=s_{jp^{k+1}}(\mathcal{P})$ for $j=1,2,\dots,q-1$. 
Then  $\mathcal{O}_{\{k+1\}}=\mathcal{P}_{\{k+1\}}$. 
\end{lemma} 

\begin{proof} 
Compare the coefficient of $x^{jp^{k+1}}$ on the two sides of \eqref{eq:G-factorization}.  Taking into account that all non-zero terms of $G_{\mathcal{O}^{[k]}}(x)$ have degree divisible by $p^{k+1}$ 
we get that 
\begin{equation}\label{eq:recursion} 
s_{jp^{k+1}}(\mathcal{O})=\sum_{i=0}^js_{ip^{k+1}}(\mathcal{O}^{[k]})s_{(j-i)p^{k+1}}(\mathcal{O}_{\{k\}}) \quad \text{ for } \quad j=1,\dots,q-1.
\end{equation}  
Consider the following system of linear equations for the unknowns $y_1,\dots,y_{q-1}$: 
\begin{equation}\label{eq:system} 
\sum_{i=1}^js_{(j-i)p^{k+1}}(\mathcal{O}_{\{k\}})y_i=s_{jp^{k+1}}(\mathcal{O})-
s_{jp^{k+1}}(\mathcal{O}_{\{k\}}), 
\quad j=1,\dots,q-1. 
\end{equation} 
By \eqref{eq:recursion} we see that $y_j=s_{jp^{k+1}}(\mathcal{O}^{[k]})$, $j=1,\dots,q-1$ is a solution of the system \eqref{eq:system}. Moreover, the system \eqref{eq:system} has a unique solution: the equation for $j=1$ gives 
\[y_1=s_{p^{k+1}}(\mathcal{O})-
s_{p^{k+1}}(\mathcal{O}_{\{k\}}),\] 
and  supposing that we have already fixed the values of $y_1,\dots,y_{j-1}$ for some $j>1$, we have 
\begin{equation*}
y_j=s_{jp^{k+1}}(\mathcal{O})-s_{jp^{k+1}}(\mathcal{O}_{\{k\}})
-\sum_{i=1}^{j-1}s_{(j-i)p^{k+1}}(\mathcal{O}_{\{k\}}) y_i. 
\end{equation*} 
Similar considerations hold for the values of  $s_{jp^{k+1}}(\mathcal{P}^{[k]})$, $j=1,\dots,q-1$. 
By assumption we have 
\[s_{ip^{k+1}}(\mathcal{P}_{\{k\}})=s_{ip^{k+1}}(\mathcal{O}_{\{k\}})
\text{ and }
s_{ip^{k+1}}(\mathcal{P})=s_{ip^{k+1}}(\mathcal{O}) 
\text{ for }i=1,\dots,q-1.\] 
It follows that $y_j=s_{jp^{k+1}}(\mathcal{P}^{[k]})$, $j=1,\dots,q-1$ is also a solution of the system \eqref{eq:system}, and by uniqueness of the solution we conclude that 
\begin{equation}\label{eq:1}
s_{jp^{k+1}}(\mathcal{O}^{[k]})=s_{jp^{k+1}}(\mathcal{P}^{[k]}) 
\text{ for }j=1,\dots,q-1.\end{equation} 
By \eqref{eq:G-p-power} we have 
\begin{equation}\label{eq:2}
s_{jp^{k+1}}(\mathcal{O}^{[k]})=s_j(\mathcal{O}^{[k]/p^{k+1}})^{p^{k+1}}
\end{equation}
and 
\begin{equation}\label{eq:3}
s_{jp^{k+1}}(\mathcal{P}^{[k]})=s_j(\mathcal{P}^{[k]/p^{k+1}})^{p^{k+1}}
\end{equation} 
By \eqref{eq:1}, \eqref{eq:2}, \eqref{eq:3} we get that 
\begin{equation}\label{eq:7} s_j(\mathcal{O}^{[k]/p^{k+1}})=s_j(\mathcal{P}^{[k]/p^{k+1}}) 
\text{ for }j=1,\dots,q-1.
\end{equation} 
We conclude from \eqref{eq:7} by Lemma~\ref{lemma:0} that 
\begin{equation}\label{eq:4}
\mathcal{O}^{[k]/p^{k+1}}_{\{0\}}=\mathcal{P}^{[k]/p^{k+1}}_{\{0\}}.
\end{equation}  
Note finally that 
\begin{equation}\label{eq:5}\mathcal{O}^{[k]/p^{k+1}}_{\{0\}}(a)=\mathcal{O}^{[k]}(a)_{k+1}=\mathcal{O}(a)_{k+1} \text{ for all }a\in \mathbb{F}_q^{\times}\end{equation}
and similarly 
\begin{equation}\label{eq:6}
\mathcal{P}^{[k]/p^{k+1}}_{\{0\}}(a)=\mathcal{P}^{[k]}(a)_{k+1}=\mathcal{P}(a)_{k+1} 
\text{ for all }a\in \mathbb{F}_q^{\times}.
\end{equation} 
Now \eqref{eq:4}, \eqref{eq:5}, \eqref{eq:6} show that 
$\mathcal{O}(a)_{k+1}=\mathcal{P}(a)_{k+1}\text{ for all }a\in \mathbb{F}_q^{\times}$. 
As we have $\mathcal{O}_{\{k\}}=\mathcal{P}_{\{k\}}$ by assumption, this gives the desired equality 
$\mathcal{O}_{\{k+1\}}=\mathcal{P}_{\{k+1\}}$.  
\end{proof} 

\begin{corollary} \label{cor:s_t} 
Let $\mathcal{O},\mathcal{P}\in \Pi_{q,n}$ and $t\in \{1,\dots,n\}$.  Assume that 
$s_{jp^k}(\mathcal{O})=s_{jp^k}(\mathcal{P})$ holds for all $j\in \{1,\dots,q-1\}$ and 
$k\in \{0,1,\dots,\lfloor \log_p(t)\rfloor\}$. 
Then we have 
\begin{itemize}
\item[(i)] $\mathcal{O}_{\{\lfloor \log_p(t)\rfloor\}}=\mathcal{P}_{\{\lfloor \log_p(t)\rfloor\}}$ 
\item[(ii)] $s_t(\mathcal{O})=s_t(\mathcal{P})$. 
\end{itemize} 
\end{corollary} 

\begin{proof} 
Lemma~\ref{lemma:0} and an iterated use of Lemma~\ref{lemma:inductive step} yield (i), so  
setting $d:=\lfloor \log_p(t)\rfloor$ we have  
$\mathcal{O}_{\{d\}}=\mathcal{P}_{\{d\}}$.
 By \eqref{eq:G-p-power} we have 
$G_{\mathcal{O}^{[d]}}(x)=G_{\mathcal{O}^{[d]/p^{d+1}}}(x)^{p^{d+1}}$, showing that all 
non-zero terms of  $G_{\mathcal{O}^{[d]}}(x)$ of positive degree have degree greater than or equal to $p^{d+1}>t$; moreover, the constant term of $G_{\mathcal{O}^{[d]}}(x)$ is $1$. On the other hand,  
by \eqref{eq:G-factorization} we have 
$G_{\mathcal{O}}(x)=G_{\mathcal{O}_{\{d\}}}(x)\cdot G_{\mathcal{O}^{[d]}}(x)$. 
It follows that $s_t(\mathcal{O})=s_t(\mathcal{O}_{\{d\}})$. 
Similarly we have $s_t(\mathcal{P})=s_t(\mathcal{P}_{\{d\}})$.  
Taking into account (i) 
we get  
\[s_t(\mathcal{O})=s_t(\mathcal{O}_{\{d\}})=s_t(\mathcal{P}_{\{d\}})=s_t(\mathcal{P}),\] 
so (ii) holds as well.   
\end{proof}

\begin{proofof}{Theorem~\ref{thm:main}}
Suppose that for the $S_n$-orbit corresponding to $\mathcal{O}\in \Pi_{q,n}$ 
and the $S_n$-orbit corresponding to $\mathcal{P}\in \Pi_{q,n}$ we have 
$s_{jp^k}(\mathcal{O})=s_{jp^k}(\mathcal{P})$ for all $j=1,\dots,q-1$ and $k\in \mathbb{Z}_{\ge 0}$ (recall that $s_m(\mathcal{O})=0=s_m(\mathcal{P})$ for any $m>n$). 
Corollary~\ref{cor:s_t} (i) in the special case $n=t$ gives   
\[\mathcal{O}_{\{\lfloor \log_p n\rfloor\}}
=\mathcal{P}_{\{\lfloor \log_p n\rfloor\}}.\] 
Taking into account that $\mathcal{O}=\mathcal{O}_{\{\lfloor \log_p n\rfloor\}}$ and 
$\mathcal{P}=\mathcal{P}_{\{\lfloor \log_p n\rfloor\}}$ we conclude the equality 
$\mathcal{O}=\mathcal{P}$. 
This clearly means that the set of elementary symmetric polynomials in the statement is separating. 
\end{proofof} 


\section{Minimality} 

\begin{lemma}\label{lemma:monotonity} 
Suppose that for some subset $A\subseteq \{1,\dots,n\}$ we have that 
$\{s_i^{(n)}\mid i\in A\}$ is separating in   $F[x_1,\dots,x_n]^{S_n}$. 
Then for all $m\le n$ we have that 
$\{s_j^{(m)}\mid j\in A\cap\{1,\dots,m\}\}$ is separating in 
$F[x_1,\dots,x_m]^{S_m}$. 
\end{lemma} 

\begin{proof}
For $n\ge m$ we shall treat $F^m$ as the subspace of $F^n$ consisting of the vectors whose last $n-m$ coordinates are zero. With this convention we have that for $v\in F^m$ and $n>m$, 
\begin{equation}\label{eq:m<n}
s_j^{(n)}(v)=\begin{cases} s_j^{(m)}(v) &\text{ for }j\le m 
\\ 0 &\text{ for }m<j\le n.\end{cases}
\end{equation} 
Now take $v,w\in F^m$ having different $S_m$-orbit. Then viewed as elements of $F^n$, 
$v$ and $w$ have different $S_n$-orbit. 
Hence by assumption there exists an $i\in A$ with $s^{(n)}_i(v)\neq s^{(n)}_i(w)$. 
In particular, $s^{(n)}_i(v)$ and $s^{(n)}_i(w)$ are not both zero, hence $i\le m$. 
Moreover, by \eqref{eq:m<n} we conclude $s^{(m)}_i(v)\neq s^{(m)}_i(w)$. 
\end{proof} 

\begin{lemma} \label{lemma:subfield} 
Let $K$ be a field containing $F$ as a subfield. 
If $\{s_i^{(n)}\mid i\in A\}$ is a separating set in  $K[x_1,\dots,x_n]^{S_n}$, 
then it is also a separating set in $F[x_1,\dots,x_n]^{S_n}$. 
\end{lemma}

\begin{proof} 
The elementary symmetric polynomials are defined over the prime subfield $\mathbb{F}_p$ of $K$ and $F$. If a set of elementary symmetric polynomials separates the $S_n$-orbits in $K^n$, then it necessarily separates the orbits in the $S_n$-stable subset $F^n$ of $K^n$. 
\end{proof} 

\begin{definition}\label{def:irreplaceable} 
We say that the elementary symmetric polynomial $s^{(n)}_k$ 
\emph{is irreplaceable over $F$} if any separating subset of $F[x_1,\dots,x_n]^{S_n}$ 
that consists of elementary symmetric polynomials necessarily contains $s^{(n)}_k$. 
\end{definition} 

\begin{remark} (i) We would like to emphasize that the question wether the $k$th elementary symmetric polynomial is irreplaceable depends both on the number of variables and  on the base field considered. 

(ii) It follows from Theorem~\ref{thm:main} that if $s^{(n)}_m$ is irreplaceable over $\mathbb{F}_q$, then $m\in [n]_q$. 
\end{remark} 

Lemma~\ref{lemma:monotonity} and Lemma~\ref{lemma:subfield} have the following immediate consequence: 

\begin{corollary} \label{cor:overfield} 
If $s^{(n)}_k$ is irreplaceable over $F$, then $s^{(m)}_k$ is irreplaceable over $K$  for all $m\ge n$ and all overfields $K$ of $F$. 
\end{corollary} 

Irreplaceable elementary symmetric polynomials have the following obvious  characterization: 

\begin{lemma} \label{lemma:characterization}
The elementary symmetric polynomial $s^{(n)}_k$  is irreplaceable over $F$ if and only if there exist elements $v,w\in F^n$ such that 
\begin{equation} \label{eq:v,w} 
s_j^{(n)}(v)=s_j^{(n)}(w) \quad \forall j\in \{1,\dots,n\}\setminus \{k\} \text{ and }s_k^{(n)}(v)\neq s_k^{(n)}(w). 
\end{equation} 
\end{lemma} 

\begin{lemma} \label{lemma:irreplaceable for all}
If $s_k^{(n)}$ is irreplaceable over $\mathbb{F}_q$, then 
$s^{(m)}_{p^jk}$ is irreplaceable over $\mathbb{F}_q$  for all $j\in\mathbb{Z}_{\ge 0}$ and $m\ge p^jn$. 
\end{lemma} 
\begin{proof}
Suppose that $s^{(n)}_k$  is irreplaceable over $\mathbb{F}_q$. Then by Lemma~\ref{lemma:characterization} 
there exist $v,w\in \mathbb{F}_q^n$ such that  $s^{(n)}_k(v)\neq s^{(n)}_k(w)$ 
and $s^{(n)}_i(v)= s^{(n)}_i(w)$ for all $i\in \{1,\dots,n\}\setminus \{k\}$. 
Denote by $\mathcal{O},\mathcal{P}\in \Pi_{q,n}$ the functions $\mathbb{F}_q\to \mathbb{Z}_{\ge 0}$ corresponding to the orbits of $v,w$, 
so we have $s_i(\mathcal{O})=s_i(\mathcal{P})$ for $i\in \{1,\dots,n\}\setminus \{k\}$ and 
$s_k(\mathcal{O})\neq s_k(\mathcal{P})$. 
Consider the polynomials 
$G_{\mathcal{O}}(x)$ and $G_{\mathcal{P}}(x)$. 
Then these are polynomials of degree at most $n$, all but their degree $k$ coefficients agree and their degree $k$ coefficient is different.  
For $m\ge p^jn$ denote by $p^j\mathcal{O}\in \Pi_{q,m}$,  $p^j\mathcal{P}\in \Pi_{q,m}$ 
the functions $a\mapsto p^j\mathcal{O}(a)$ ($a\in \mathbb{F}_q^{\times}$), 
$a\mapsto   p^j\mathcal{P}(a)$ ($a\in \mathbb{F}_q^{\times}$). 
We have the equalities 
\[G_{p^j\mathcal{O}}(x)=G_{\mathcal{O}}(x)^{p^j}=\sum_{i=0}^n
s_i(\mathcal{O})^{p^j}x^{ip^j}
\text{ and } 
G_{p^j\mathcal{P}}(x)=G_{\mathcal{P}}(x)^{p^j}=\sum_{i=0}^n
s_i(\mathcal{P})^{p^j}x^{ip^j}.\] 
This shows that $s_{p^jk}(p^j\mathcal{O})\neq s_{p^jk}(p^j\mathcal{P})$ and 
$s_i(p^j\mathcal{O})= s_i(p^j\mathcal{P})$ for all $i\in \{1,\dots,m\}\setminus \{k\}$. 
Consequently, $s^{(m)}_{p^jk}$ is irreplaceable over $\mathbb{F}_q$ by Lemma~\ref{lemma:characterization}. 
\end{proof}

\begin{corollary} \label{cor:minimality} 
Suppose that the elementary symmetric polinomial $s^{(k)}_k$ is irreplaceable over 
$\mathbb{F}_q$ for each $k\in \{1,2,\dots,q-1\}$. 
Then for an arbitrary $n$ the separating subset 
$\{s^{(n)}_{m}\mid  m\in [n]_q\}$ 
of $\mathbb{F}_q[x_1,\dots,x_n]^{S_n}$ given in Theorem~\ref{thm:main} is minimal (with respect to inclusion).
\end{corollary}  

\subsection{Some irreplaceable elementary symmetric polynomials}\label{subsec:some irreplaceables}   

Below for certain prime powers $q$ and certain integers $n,k$  we 
provide pairs of vectors $v,w$ in $\mathbb{F}_q^n$ satisfying \eqref{eq:v,w},  
showing by Lemma~\ref{lemma:characterization} that 
$s_k^{(n)}$  is irreplaceable over $\mathbb{F}_q$. 
The elements of the $p$-element field $\mathbb{F}_p$ will be denoted by $0,1,2,\dots,p-1$ in the obvious way. 

\[\begin{array}{c||c|c} 
q=3 & s_1^{(1)} & s_2^{(2)} 
 \\ \hline \hline 
v & [1] & [1,2]   
\\ \hline
w & [0] & [0,0] 
\end{array}\] 

\[\begin{array}{c||c|c|c|c} 
q=5 & s_1^{(1)} & s_2^{(2)} & s_3^{(3)} & s_4^{(4)}  
 \\ \hline \hline 
v & [1] & [1,4] & [1,4,4] & [1,2,3,4]  
\\ \hline
w & [0] & [0,0] & [2,2,0] & [0,0,0,0]  
\end{array}\] 

\[\begin{array}{c||c|c|c|c|c|c} 
q=7 & s_1^{(1)} & s_2^{(2)} & s_3^{(3)} & s_4^{(4)} & s_6^{(6)} & s_5^{(6)} 
 \\ \hline \hline 
v & [1] & [1,6] & [1,2,4] & [1,1,6,6]  &  [1,2,3,4,5,6] & [2,2,2,3,6,6]
\\ \hline
w & [0] & [0,0] & [0,0,0] & [3,4,0,0]  &  [0,0,0,0,0,0] &  [1,1,4,5,5,5]
\end{array}\]  

\[\begin{array}{c||c|c|c|c|c|c} 
q=11 & s_1^{(1)} & s_2^{(2)} & s_3^{(3)} & s_4^{(4)} & s_5^{(5)} & s_6^{(6)} 
 \\ \hline \hline 
v & [1] & [5,6] & [5,7,9] & [1,2,2,5]  &  [1,3,9,5,4] & [1,1,3,6,8,8]
\\ \hline
w & [0] & [0,0] & [10,0,0] & [10,0,0,0]  &  [0,0,0,0,0,0] &  [9,9,10,10,0,0]
\end{array}\] 

\[\begin{array}{c||c} 
q=11 & s_7^{(8)} 
 \\ \hline \hline 
v & [1, 2, 2, 2, 2, 4, 4, 5] 
\\ \hline
w & [6, 7, 7, 9, 9, 9, 9, 10]
\end{array}\] 

\subsection{The polynomial $s_k^{(k)}$ for a divisor $k$ of $q-1$ for an arbitrary prime power $q$}  
\label{subsec:k|p-1} 
We may take as 
$v$ the vector whose coordinates are the $k$ roots in  $\mathbb{F}_q$ 
of the polynomial $x^k-1$, and for $w$ the zero vector. 
Then we have $0=s_1^{(k)}(w)=s_1^{(k)}(v)=s_2^{(k)}(w)=s_2^{(k)}(v)=
\cdots =s_{k-1}^{(k)}(w)=s_{k-1}^{(k)}(v)=s_k^{(k)}(w)$ 
whereas $s_k^{(k)}(v)=(-1)^{k+1}\in \mathbb{F}_q$.

\subsection{The polynomial $s_k^{(k)}$ for an arbitrary $k$ and $q\ge k!-k+1$} 
 The assumption on $q$ guarantees that the number $\binom{q-1+k}{k}$ of $S_k$-orbits in $\mathbb{F}_q^k$ is greater than the number $q^{k-1}$ which is an obvious upper bound for the number of possible values of the 
 $(k-1)$-tuples $(s_1^{(k)}(z),\dots,s_{k-1}^{(k)}(z))$ (where $z\in \mathbb{F}_q$), hence the desired pair of vectors in $\mathbb{F}_q^k$ exists by the pigeonhole principle. 
 From the special case $k=3$ we get the $s_3^{(3)}$ 
 is irreplaceable over $\mathbb{F}_q$ for all prime powers  
 $q\ge 4$. 

\subsection{Result obtained by  computer for $q=7$.}\label{subsec:p=7} 
$\{s_k^{(5)}\mid k=1,2,3,4\}$ is a separating set in 
$\mathbb{F}_7[x_1,x_2,x_3,x_4,x_5]^{S_5}$, hence 
the elementary symmetric polynomial 
$s_5^{(5)}$ is not irreplaceable over $\mathbb{F}_7$. 

\subsection{Further results obtained by  computer for $q=11$.}\label{subsec:p=11} 
The elementary symmetric polynomial  
$s_7^{(7)}$ is not irreplaceable over $\mathbb{F}_{11}$. On the other hand,  
$\{s_i^{(10)}\mid i=1,2,3,4,5,6,7, 10\}$ is a minimal separating set in $\mathbb{F}_{11}[x_1,\dots,x_{10}]^{S_{10}}$.  Consequently,  none of 
$s_8^{(8)}, s_8^{(9)}, s_8^{(10)},s_9^{(9)},s_9^{(10)}$ is irreplaceable over $\mathbb{F}_{11}$.  
Further computer calculations showed that  none of 
$s_8^{(11)}$, $s_8^{(12)}$, $s_8^{(13)}$  
is irreplaceable over $\mathbb{F}_{11}$.  
 
\bigskip 
The results of Section~\ref{subsec:some irreplaceables} (and Section~\ref{subsec:p=7})  imply by Lemma~\ref{lemma:irreplaceable for all} and Corollary~\ref{cor:minimality} 
the following: 

\begin{corollary}\label{cor:p=3,5,7} The separating subset 
$\{s_m^{(n)}\mid m\in [n]_q\}$ of $\mathbb{F}_q[x_1,\dots,x_n]^{S_n}$ given in Theorem~\ref{thm:main} is minimal (with respect to inclusion) for 
$q=3$, $4$, $5$ with arbitrary $n$, and for $q=7$ with 
$\log_7 n-\lfloor \log_7 n\rfloor<\log_7 5$ or $\log_7n-\lfloor \log_7 n\rfloor\ge\log_7 6$. 
In $\mathbb{F}_7[x_1,x_2,x_3,x_4,x_5]^{S_5}$,  $\{s_i^{(5)}\mid i=1,2,3,4\}$ is a minimal separating subset.  
\end{corollary} 

In view of the above results it seems natural to formulate the following problems: 

\begin{problem} 
Is it true that for any prime $p$ and any $k\in \{1,2,\dots,p-1\}$ there exists a positive integer $n$ such that  $s_k^{(n)}$ is irreplaceable over $\mathbb{F}_p$? 
\end{problem} 

\begin{problem} 
Given a prime $p$ and positive integers $k\le n$, determine the minimal fields $\mathbb{F}_q$ of characteristic $p$ such that $s_k^{(n)}$ is irreplaceable over $\mathbb{F}_q$. 
\end{problem} 

\subsection{A bound for the distance from being minimal} 

We saw above that the separating set given in Theorem~\ref{thm:main} is not always  minimal. Moreover, even when it is minimal, it may not be of minimal possible cardinality. 
For example, for $n=9$ and $q=3$, the minimal cardinality of a separating set in 
$\mathbb{F}_3[x_1,\dots,x_9]^{S_3}$ is $\lceil \log_3\binom{9+2}{2}\rceil=4$ 
by \cite[Theorem 1.1]{KLR}, whereas 
the minimal separating set given in Corollary~\ref{cor:p=3,5,7} has cardinality 
$|[9]_3|=|\{1,2,3,6,9\}|=5$. Our aim here is to point out that however, for $q=p$ and for large $n$, the separating set given in Theorem~\ref{thm:main} is not much bigger than a separating set of minimal cardinality. 

Denote by $d_q(n)$ the difference of $|[n]_q|$ (the number  of elements in the separating subset 
of $\mathbb{F}_q[x _1,\dots,x_n]^{S_n}$ given in Theorem~\ref{thm:main}) 
and the number of elements in a separating subset of  
$\mathbb{F}_q[x _1,\dots,x_n]^{S_n}$ of minimal possible cardinality. 

\begin{proposition}\label{prop:d(n)} 
For any prime $p$ we have the inequality $d_p(n)\le p-2$. Consequently, one can get a minimal (with respect to inclusion) separating set in $\mathbb{F}_p[x_1,\dots,x_n]^{S_n}$ by removing at most $p-2$ elements from the separating set given in Theorem~\ref{thm:main}. 
\end{proposition} 

\begin{proof} By \cite[Theorem 1.1]{KLR} the minimal cardinality of a separating set in $\mathbb{F}_p[x_1,\dots,x_n]^{S_n}$ is the upper integer part of the logarithm with base $p$ of the number of $S_n$-orbits in $\mathbb{F}_p^n$. 
The number of $S_n$-orbits in $\mathbb{F}_p^n$ is $\binom{n+p-1}{p-1}$. 
There exists an $m\in\{1,\dots,p-1\}$ and $k\in \mathbb{Z}_{\ge 0}$ with 
\[mp^k\le n<(m+1)p^k.\] 
We have 
\begin{align}\label{eq:log p}  
\lceil\log_p\binom{n+p-1}{p-1}\rceil&\ge
\lceil\sum_{j=1}^{p-1}\log_p(mp^k+j)-\log_p((p-1)!)\rceil
\\ \notag &>(p-1)\log_p(mp^k)-\log_p((p-1)!)
\\ \notag &=(p-1)k+(p-1)\log_pm-\log_p((p-1)!)
\\ \notag &>(p-1)k+m-(p-1),  
\end{align} 
where the last inequality follows from 
\[\frac{(p-1)!}{m^{p-1}}=\frac{p-1}{m}\cdot \frac{p-2}{m}\cdots \frac{m+1}{m}\cdot \frac{m}{m}\cdot \frac{m-1}{m}\cdots \frac{1}{m}<p^{p-1-m}.\] 
On the other hand, the number of elements in the separating set given in Theorem~\ref{thm:main} is $k(p-1)+m$. Taking into account \eqref{eq:log p} we get 
\[d_pn)<k(p-1)+m-((p-1)k+m-(p-1))=p-1.\] 
\end{proof}


\section{Multisymmetric polynomials over $\mathbb{F}_q$} \label{sec:multisymmetric}

Consider the $m$-fold direct sum of the representation of $S_n$ on $F^n$. So the underlying vector space of this representation is 
$(F^n)^m=F^n\oplus\cdots\oplus F^n$. For $j=1,\dots,m$ and $i=1,\dots,n$ denote by $x_i^{(j)}$ the function mapping an $m$-tuple $(v^{(1)},\dots,v^{(m)})\in (F^n)^m$ of vectors to the $i$th coordinate of the $j$th vector component $v^{(j)}$. We get an induced action on the $nm$-variable polynomial algebra $A_{n,m}:=F[x_i^{(j)}\mid i=1,\dots,n;\ j=1,\dots,m]$ given by 
$\pi \cdot x_i^{(j)}=x_{\pi(i)}^{(j)}$.  The corresponding algebra 
$A_{n,m}^{S_n}$ of polynomial invariants is called 
\emph{the algebra of multisymmetric polynomials}. 
Our aim in this section is to give a separating set of multisymmetric polynomials, where a subset $T$ of $A_{n,m}^{S_n}$ is said to be \emph{separating} if for any $v,w\in (F^n)^m$ with different $S_n$-orbit there is a polynomial 
$f\in T$ with $f(v)\neq f(w)$.  

The algebra $A_{n,m}$ is $\mathbb{Z}_{\ge 0}^m$-graded: the multihomogeneous component 
$A_{n,m,\alpha}$ of $A_{n,m}$ of multidegree $\alpha=(\alpha_1,\dots,\alpha_m)\in\mathbb{Z}_{\ge 0}^m$ is spanned by polynomials all of whose non-zero terms have total degree $\alpha_j$ in the variables $x_1^{(j)},\dots,x_n^{(j)}$ for each $j=1,\dots,m$. The $S_n$-action preserves this multigrading, hence the algebra of multisymmetric polynomials is also multigraded: we have 
$A_{n,m}^{S_n}=\bigoplus_{\alpha\in\mathbb{Z}_{\ge 0}^m}A_{n,m,\alpha}^{S_n}$. 
Denote by $s_{k,\alpha}^{(n)}$ the component of multidegree $\alpha$ of 
$s_k^{(n)}(\sum_{j=1}^mx_1^{(j)},\dots,\sum_{j=1}^mx_n^{(j)})$. Clearly 
$s_{k,\alpha}^{(n)}$ is non-zero only if $\alpha_1+\cdots+\alpha_m=k$.  
The multisymmetric polynomials $s_{k,\alpha}^{(n)}$ are called the \emph{polarizations of the elementary symmetric polynomials}. They generate $A_{n,m}^{S_n}$ when the characteristic of the base field $F$ is greater than $n$ or is $0$ (in other words, when 
$\mathrm{char}(F)$ does not divide the order of the group $S_n$, see for example \cite{domokos} and the relevant references therein). 
However, when $0<\mathrm{char}(F)\le n$, the polarizations of the elementary symmetric polynomials are not suficient to generate the algebra of multisymmetric polynomials in general. In fact the maximal degree of an element in a minimal homogeneous generating 
system of $A_{n,m}^{S_n}$ tends to infinity together with $m$, see \cite{fleischmann}.    
For the modular case generating systems of $A_{n,m}^{S_n}$ are given in \cite{vaccarino}, 
\cite{domokos}.   A minimal homogeneous generating system is obtained in \cite{feshbach} 
for the case $F=\mathbb{F}_2$,  and in \cite{rydh} for an arbitrary base ring $F$. 

\begin{proposition}\label{prop:DKW} 
For any field $F$,  the following is a separating set in $A_{n,m}^{S_n}$:  
\[ \left\{\sum_{\alpha_2+2\alpha_3+\cdots+(m-1)\alpha_m=d}s_{k,\alpha}^{(n)} \ \mid \ 
k=1,\dots,n; \quad d=0,1,\dots,(m-1)n\right\}.\] 
\end{proposition} 
\begin{proof} 
The elementary symmetric polynomials generate 
$F[x_1,\dots,x_n]^{S_n}$, therefore by \cite[Theorem 3.4]{draisma-kemper-wehlau} their ``cheap polarizations" form a separating set. It is easy to see that the cheap polarizations of the elementary symmetric polynomials are the polynomials given in the statement. 
\end{proof}  

Let us recall another class of multisymmetric polynomials. For any $\alpha\in \mathbb{Z}_{\ge 0}^m$ set 
\[s_k^{(n)}(x^{\alpha}):=s_k^{(n)}(\prod_{j=1}^m{x_1^{(j)}}^{\alpha_j},\dots,
\prod_{j=1}^m{x_n^{(j)}}^{\alpha_j}).\] 
Note that the formulae from \cite{amitsur} were used in \cite[Equation (6)]{domokos} to express the polynomials $s_{k,\alpha}^{(n)}$ in terms of the polynomials $s_k^{(n)}(x^{\gamma})$. 
Write $|\alpha|:=\sum_{j=1}^m\alpha_j$,  
and 
denote by $\mathrm{gcd}(\alpha)$ the greatest common divisor of $\alpha_1,\dots,\alpha_m$. 

\begin{proposition}\label{prop:amitsur} 
For any field $F$ the elements $s_k^{(n)}(x^{\alpha})$ with 
$\alpha\in \mathbb{Z}_{\ge 0}^m$, $k|\alpha|\le n$, $\mathrm{gcd}(\alpha)=1$ 
constitute a separating set in $A_{n,m}^{S_n}$. 
\end{proposition} 
\begin{proof}  
Proposition~\ref{prop:DKW} implies in particular  that the elements of $A_{n,m}^{S_n}$ with degree at most $n$ form a separating subset. It follows that the elements of degree at most $n$ from any homogeneous system of generators of the algebra $A_{n,m}^{S_n}$ form a separating set. Applying this for the generating system given in \cite[Corollary 5.3]{domokos} 
we obtain the desired statement. 
\end{proof} 

\begin{theorem}\label{thm:multisymmetric} 
For the field $F=\mathbb{F}_q$ of $q$ elements the following is a separating set in $A_{n,m}^{S_n}$: 
\begin{align}\label{eq:multisymm sep}\{s_{jp^k}^{(n)}(x^{\alpha})\mid \quad & j\in \{1,\dots,q-1\},\quad \alpha\in \mathbb{Z}_{\ge 0}^m,\quad |\alpha|\le n,\quad \mathrm{gcd}(\alpha)=1,
\\ \notag  & \alpha_j\le q-1 \text{ for }j=1,\dots,m, 
\quad k\in \{0,1,\dots,\lfloor \log_p\frac{n}{|\alpha|}\rfloor\}\}.
\end{align} 
\end{theorem} 

\begin{proof} 
Suppose that $v,w\in (\mathbb{F}_q^n)^m$ belong to different $S_n$-orbits. 
By Proposition~\ref{prop:amitsur} there exist a $t\in\{1,\dots,n\}$, 
$\alpha\in \mathbb{Z}_{\ge 0}^m$ with $t|\alpha|\le n$ such that 
$s_t^{(n)}(v^{\alpha})\neq s_t^{(n)}(w^{\alpha})$, where 
$v^{\alpha}$ (respectively $w^{\alpha}$) stands for the vector in $\mathbb{F}_q^n$ whose $i$th coordinate is 
$\prod_{j=1}^m{v_i^{(j)}}^{\alpha_j}$ (respectively $\prod_{j=1}^m{w_i^{(j)}}^{\alpha_j}$).  
We may assume that $|\alpha|$ is minimal possible. 
Then $\alpha_j\le q-1$, because for all $j$ with $\alpha_j>0$, denoting by 
$\gamma_j$ the unique element of $\{1,\dots,q-1\}$ which is congruent to $\alpha_j$ modulo $q-1$ (and setting $\gamma_j=0$ if $\alpha_j=0$), we have 
$v^{\alpha}=v^{\gamma}$ and $w^{\alpha}=w^{\gamma}$. 
Thus by minimality of $|\alpha|$ we must have $\alpha_j\in \{0,1,\dots,q-1\}$ for al $j$. 
Moreover, we must have $\mathrm{gcd}(\alpha)=1$, since otherwise $s_t^{(n)}(x^{\alpha})$ can be expressed as a polynomial of elements of the form $s_u^{(n)}(x^{\gamma})$ with $|\gamma|< |\alpha|$ (see \cite[Page 517]{domokos} for explanation), and some $s_u^{(n)}(x^{\gamma})$ with $|\gamma|<|\alpha|$ would separate $v$ and $w$, contrary to the minimality of $|\alpha|$. 
Finally, it follows by Corollary~\ref{cor:s_t} (ii) that there exists a $j\in \{1,\dots,q-1\}$ and 
$k\in \{0,1,\dots,\lfloor \log_pt\rfloor\}$ such that 
$s_{jp^k}^{(n)}(v^{\alpha})\neq s_{jp^k}^{(n)}(w^{\alpha})$. 
So we showed that whenever  $v,w\in (\mathbb{F}_q^n)^m$ have different $S_n$-orbit, then 
$v$ and $w$ can be separated by an element from \eqref{eq:multisymm sep}. 
\end{proof} 

\begin{remark} 
The separating set given in Theorem~\ref{thm:multisymmetric} for $F=\mathbb{F}_q$ exploits Theorem~\ref{thm:main}. For fixed $q$ and $m$ and ``sufficiently large $n$"  it is significantly smaller than the separating sets given in 
Proposition~\ref{prop:DKW} or Proposition~\ref{prop:amitsur} for general $F$, see 
Example~\ref{example:26} for illustration. 
On the other hand, in the special case $q=2$ a stronger result is known, since 
in \cite[Theorem 4.8]{KLR} a minimal separating subset of $A_{n,m}^{S_n}$ is determined 
for $F=\mathbb{F}_2$; this separating set is a proper subset of the one given by the special case $q=2$ of our Theorem~\ref{thm:multisymmetric}, as it involves a stronger upper bound for the parameter $k$ in the multisymmetric polynomials included in the separating set.  
\end{remark} 

\begin{example} \label{example:26} 
\begin{itemize} 
\item[(i)] Take $q=3$, $m=2$, and $n=26$. The possible $\alpha$ to consider in the separating set 
\eqref{eq:multisymm sep} in Theorem~\ref{thm:multisymmetric} for 
$A_{26,2}^{S_{26}}$  are $\alpha=(1,0)$, $(0,1)$, $(1,1)$, $(2,1)$, $(1,2)$. 
The corresponding numbers $\lfloor \log_3\frac{26}{|\alpha|}\rfloor$ are 
$2,2,2,1,1$. Thus in this case the separating set \eqref{eq:multisymm sep} has 
$2\cdot (3+3+3+2+2)=26$ elements. The separating set of $A_{n,m}^{S_n}$ given in 
Proposition~\ref{prop:DKW} has $n(n(m-1)+1)$ elements in general, so it has 
$26\cdot 27=702$ elements in our case $n=26$, $m=2$. 

On the other hand, the number of $S_{26}$-orbits in $\mathbb{F}_3^{26}\oplus \mathbb{F}_3^{26}$ is $\binom{34}{8}=18156204$. 
Thus by  \cite[Theorem 1.1]{KLR}, there is a separating set in $A_{26,2}^{S_{26}}$ 
consisting of $\lceil \log_3\binom{34}{8}\rceil=16$ multisymmetric polynomials. 
\item[(ii)] Take $q=3$, $m=2$, and $n=8$. 
The separating set \eqref{eq:multisymm sep} for $A_{8,2}^{S_8}$ in Theorem~\ref{thm:multisymmetric} has $16$ elements, the separating set given in 
Proposition~\ref{prop:amitsur} has $44$ elements, whereas the separating set given in 
Proposition~\ref{prop:DKW} has $72$ elements. The minimal cardinality of a separating set for 
$A_{8,2}^{S_8}$ over $\mathbb{F}_3$ is $9$. 
\end{itemize} 
\end{example} 

\section{Comment on lacunary polynomials}

Lemma~\ref{lemma:characterization} has the following reformulation in terms of polynomials:  
The elementary symmetric polynomial $s_k^{(n)}\in F[x_1,\dots,x_n]$ is irreplaceable if there exist two polynomials $f,g\in F[x]$ having degree at most $n$ and constant term $1$ such that both $f$ and $g$ split as a product of linear factors over $F$, the degree $k$ coefficient of $f$ differs from the degree $k$ coefficient of $g$, and all the other coefficients of $f$ coincide with the corresponding coefficient of $g$. This is reminiscent of the topic of the book \cite{redei}, where lacunary polynomials that split as a product of root factors over $\mathbb{F}_q$ are studied. 

In particular, our 
Corollary~\ref{cor:i_1=0} has similar flavour as the following theorem of R\'edei 
\cite[Paragraph 10]{redei}: Suppose that the polynomial $f(x)=x^q+g(x)$ splits as a product of root factors over the field $\mathbb {F}_q$, and the formal derivative of $f$ is non-zero. Then $\deg(g)\ge (q+1)/2$ or $f(x)=x^q-x$. For applications of this result in finite geometry see \cite{szonyi}.


\end{document}